\documentclass[a4paper, 12pt]{article}

\usepackage[utf8]{inputenc}
\usepackage[T1]{fontenc}
\usepackage[a4paper, margin=2.9cm]{geometry}
\usepackage{needspace}

\usepackage{libertine} 
\usepackage{inconsolata} 

\usepackage{amsmath, amsthm, amssymb}
\usepackage{graphicx}
\usepackage{enumerate}
\usepackage{authblk}
\usepackage[textsize=small, textwidth=2cm, color=yellow]{todonotes}
\usepackage{thmtools}
\usepackage{thm-restate}
\usepackage[colorlinks=true, citecolor=red]{hyperref}
\usepackage{float}

\usepackage{todonotes}
\usepackage{esvect}
\usepackage{cleveref}
\usepackage{tikz}
\usetikzlibrary{arrows}

\makeatletter
\newcommand\@erelb@r[1]{%
	\mathrel{\tikz[baseline=-.5ex]\draw[#1] (0,0)--(0.3,0);}
}

\newcommand{\erelbar}[1]{\@erelbar#1}
\def\@erelbar#1#2{%
	\ifcase\numexpr#1*4+#2\relax
	\@erelb@r{|-}\or 
	\@erelb@r{-|}   
	\else
	\@wrong
	\fi
}
\makeatother

\declaretheorem[name=Theorem]{theorem}
\declaretheorem[name=Lemma, sibling=theorem]{lemma}

\declaretheorem[name=Conjecture, sibling=theorem]{conjecture}
\declaretheorem[name=Problem, sibling=theorem]{problem}

\def\cqedsymbol{\ifmmode$\lrcorner$\else{\unskip\nobreak\hfil
		\penalty50\hskip1em\null\nobreak\hfil$\lrcorner$
		\parfillskip=0pt\finalhyphendemerits=0\endgraf}\fi}

\let\le\leqslant
\let\ge\geqslant

\title{Polynomial Gyárfás-Sumner conjecture for graphs of bounded boxicity}

\author[1]{James Davies}
\author[2]{Yelena Yuditsky}

\affil[1]{University of Cambridge, United Kingdom.}
\affil[2]{Université libre de Bruxelles, Belgium.}
\date{}

\begin{document}
	\sloppy
	
	\maketitle
	
	\begin{abstract}
		We prove that for every positive integer $d$ and forest $F$, the class of intersection graphs of axis-aligned boxes in $\mathbb{R}^d$ with no induced $F$ subgraph is (polynomially) $\chi$-bounded.
	\end{abstract}
	
	\section{Introduction}

	For a graph $G$, we denote its clique and chromatic number by $\omega(G)$ and $\chi(G)$ respectively.
	A class of graphs $\mathcal{G}$ is \emph{(polynomially) $\chi$-bounded} if there exists a (polynomial) function $f:\mathbb{N} \to \mathbb{N}$ such that $\chi(G)\le f(\omega(G))$ for every $G\in \mathcal{G}$.
	Note that not all classes of graphs are $\chi$-bounded~\cite{tutte,descartes1954solution} and not all hereditary $\chi$-bounded classes of graphs are polynomially $\chi$-bounded~\cite{brianski2024separating}.
	
	The Gyárfás-Sumner conjecture \cite{gyarfas1975ramsey,sumner1981subtrees} states that for every forest $F$, the class of graphs with no induced $F$ subgraph is $\chi$-bounded.
	The beauty of the Gyárfás-Sumner conjecture is that forests $F$ are the only graphs for which such a class could be $\chi$-bounded, since Erd\H{o}s~\cite{erdos1959graph} show that there are graphs with arbitrarily large girth and chromatic number.
	R\"{o}dl (see~\cite{gyarfas1980induced,kierstead1996applications}) proved a natural weakening of the Gyárfás-Sumner conjecture that graphs without an induced $K_{s,s}$ and an induced forest $F$ are $\chi$-bounded.
	Scott, Seymour and Spirkl~\cite{scott2023polynomialI} further proved that such graphs are polynomially $\chi$-bounded.
	R\"{o}dl's theorem retains the same beauty of the Gyárfás-Sumner conjecture since graphs without an induced $K_{s,s}$ clearly still include graphs with arbitrarily large girth and chromatic number.
	
	One of the most classical results in $\chi$-boundedness from 1960 is that rectangle intersection are (polynomially) $\chi$-bounded~\cite{asplund1960coloring}.
	This is not the case for boxes in $\mathbb{R}^3$ since in 1965, Burling~\cite{burling1965coloring} showed that there are triangle-free intersection graphs of axis-aligned boxes in $\mathbb{R}^3$ with arbitrarily large chromatic number.
	The first author~\cite{davies2021box} further showed that there are box intersection graphs with arbitrarily large girth and chromatic number.
	We prove the Gyárfás-Sumner conjecture for box intersection graphs and more generally for intersection graphs of axis-aligned boxes in $\mathbb{R}^d$.
	We even obtain polynomial $\chi$-boundedness.
	
	\begin{theorem}\label{main}
		For every positive integer $d$ and forest $F$, the class intersection graphs of axis-aligned boxes in $\mathbb{R}^d$ with no induced $F$ subgraph is polynomially $\chi$-bounded.
	\end{theorem}
	
	Two other other classes of geometric graphs that are polynomially $\chi$-bounded are circle graphs~\cite{davies2022improved,davies2021circle} and (more generally) grounded $L$-graphs~\cite{davies2023grounded}.
	Two more geometric $\chi$-bounded classes of graphs are outerstring graphs~\cite{rok2019outerstring} (which generalize circle graphs and grounded $L$-graphs) and polygon visibility graphs~\cite{davies2023coloring}.
	
	The Gyárfás-Sumner conjecture is widely open and is only known for forests $F$ with relatively simple components~\cite{chudnovsky2019induced,gyarfas1987problems,kierstead2004radius,kierstead1994radius,scott2020induced,scott1997induced} such as paths~\cite{gyarfas1987problems}, subdivided stars~\cite{scott1997induced} and trees of radius at most two~\cite{kierstead1994radius}.
	Polynomial $\chi$-boundedness is only known for quite simple forests~\cite{chudnovsky2023polynomial,liu2023polynomial,scott2022polynomialII,scott2022polynomial}, the most general result being for forests $F$ where at most one component is a double star and the rest are stars or $P_4$~\cite{chudnovsky2023polynomial,scott2022polynomialII,scott2022polynomial}. A quasi-polynomial bound is known for graphs without an induced $P_5$~\cite{scott2023polynomial}.
	
	One other class of graphs that is not $\chi$-bounded~\cite{pawlik2014triangle} where the Gyárfás-Sumner conjecture holds is string graphs~\cite{chudnovsky2021induced}. Although unlike box intersection graphs and graphs without an induced $K_{s,s}$, string graphs with girth at least 5 have bounded chromatic number~\cite{kostochka1998coloring}.
	Like with considering the Gyárfás-Sumner conjecture for specific forests $F$, studying the Gyárfás-Sumner conjecture for more specific classes of graphs could develop techniques that may eventually contribute to resolving the Gyárfás-Sumner conjecture.
	For more on $\chi$-boundedness, see the survey of Scott and Seymour~\cite{scott2018survey}.
	
	In \Cref{sec:open} we shall discuss further problems.
	In the next two sections, we prove \Cref{main}. \Cref{sec:direct} is dedicated to a preliminary lemma on directed graphs, and then \Cref{sec:boxes} focuses on box intersection graphs.

	\section{Directed graphs}\label{sec:direct}
	
	This section is dedicated to proving a purely graph theoretic lemma on finding ``path-induced'' directed trees in certain directed graphs with large chromatic number (see \Cref{directedmain}). We remark that Nguyen, Scott, and Seymour~\cite{nguyen2024note} proved a path-induced weakening of the Gyárfás-Sumner conjecture, however this is insufficient for our purposes since we require rooted directed trees.

	For a directed graph $\overrightarrow{D}$, we let $D$ be its underlying undirected graphs.
	For a vertex $u\in V(D)$, we let $N^+(u)$ denote the set of vertices $v$ such that $\overrightarrow{D}$ has a directed edge from $u$ to $v$. These are the \emph{out-neighbours} of $u$. When the directed graph $\overrightarrow{D}$ is not clear from context, we use $N^+_{\overrightarrow{D}}(u)$.
	We say that a set of vertices in $\overrightarrow{D}$ is a clique if it is a clique in $D$.
	A directed tree $\overrightarrow{T}$ is \emph{rooted} if there exists a (root) vertex $r$ such that for every leaf $\ell$ of $T$, the path between $r$ and $\ell$ is a directed path from $r$ to $\ell$.

	For a pair of vertices $u,v$ of a directed graph $\overrightarrow{D}$, we let $t(u,v)$ denote the maximum size of a \emph{transitive tournament} starting at $u$ and ending at $v$, where a \emph{transitive tournament} is a complete acyclic directed graph. 
	
	For positive integers $k,m$, a \emph{$(k,m)$-grading} of a directed graph $\overrightarrow{D}$ is collection $(X_1, \ldots , X_m)$ where $X_1\subseteq \cdots \subseteq X_m = V(D)$ with $X_1$ non-empty such that for every $1 \le i < m$ and $v \in X_i$, $|N^+(v)\cap X_{i+1}|\ge k$.
	For a directed graph $\overrightarrow{D}$ with a given grading $(X_1, \ldots , X_m)$, for each $v\in V(D)$, we let $g(v)$ be equal to the minimum $1\le i \le m$ such that $v\in X_i \backslash X_{i-1}$ (where $X_{0}=\emptyset$).
	For a vertex $v$ of a rooted directed tree $\overrightarrow{T}$, we let $\overrightarrow{T_v}$ be the rooted directed tree subgraph consisting of $v$ and its \emph{decedents} (the vertices $w$ for which there is is a directed path from $v$ to $w$ in $\overrightarrow{T}$).
	A rooted directed tree subgraph $\overrightarrow{T}$ of a directed graph $\overrightarrow{D}$ with grading $(X_1, \ldots , X_m)$ is \emph{calm} if its root is contained in $X_1$ and for every $uv\in E(\overrightarrow{T})$, we have that $t(u,v)-1\ge g(v)-g(u)$, and for every $uw\in E(\overrightarrow{D})$ with $w\in V(D)\backslash (V(T) \backslash V(T_v)) $ with $t(u,w)-1\ge g(w)-g(u)$, we have that $t(u,v)\ge t(u,w)$.
	
	We prove that within large gradings, we can find calm copies of trees.
	The \emph{depth} of a rooted directed tree $\overrightarrow{T}$ is equal to longest length of a directed path from the root to a leaf.
	
	\begin{lemma}\label{grading}
		Let $\overrightarrow{T}$ be a rooted directed tree with depth $d$ and $n$ vertices and let $D$ be a directed graph with a $(k,m)$-grading $(X_1 , \ldots , X_m)$ where $m \ge 1 + (d-1)(\omega(D)-1)$ and $k \ge n-1$.
		Then there is a calm copy $\overrightarrow{T'}$ of $\overrightarrow{T}$ in the grading $(X_1 , \ldots , X_m)$.
	\end{lemma}
	
	\begin{proof}
		Clearly the lemma is true if $n=1$, so we shall proceed inductively on $n$.
		Let $p_0\cdots p_\ell$ be a directed path of $\overrightarrow{T}$ where $p_0$ is the root and $p_\ell$ is a leaf.
		Note that $\ell\le d$.
		By the inductive hypothesis, there is a calm copy $\overrightarrow{T^*}$ of $\overrightarrow{T} - p_\ell$ in $(X_1 , \ldots , X_m)$.
		For each $v\in V(T) - p_\ell$, let $v'$ be the corresponding vertex of $\overrightarrow{T^*}$ in $\overrightarrow{D}$.
		
		For every $u'v'\in E(\overrightarrow{T^*})$, we have that $g(v')-g(u')\le t(u',v') - 1 \le \omega(D) - 1$.
		Therefore $g(p_{\ell -1}')\le (d-1)(\omega(D) - 1)$.
		Since $(X_1,\ldots , X_m)$ is a $(k,m)$-grading of $\overrightarrow{D}$ and $k\ge n-1$, there is some vertex $w\in X_{g(p_{\ell -1}')+1} \backslash V(T^*)$ with $p'_{\ell -1}w \in E(\overrightarrow{D})$.
		In particular, $t(p'_{\ell -1},w)  - 1\ge 1 \ge g(w) - g(p'_{\ell -1})$.
		Now, let $p'_\ell \in V(D)\backslash V(T^*)$ be a vertex with $p'_{\ell -1}p'_{\ell }\in E(\overrightarrow{D})$, $t(p'_{\ell -1},p'_\ell)  - 1\ge g(p'_\ell ) - g(p'_{\ell -1})$, and subject to this, $t(p'_{\ell -1},p'_\ell)$ is maximised.
		Note that $w$ shows that such a vertex $p'_\ell$ exists.
		Then we can obtain our desired tree $T'$ by adding $p_\ell'$ and $p'_{\ell -1}p'_\ell$ to $T^*$.
	\end{proof}
	
	The next lemma essentially allows us to find large gradings. The exact statement better facilities its inductive proof.
	For a set of vertices $C$ in a graph $G$, we let $G[C]$ denote the induced subgraph of $G$ on vertex set $C$.
	
	\begin{lemma}\label{highDegree}
		Let $c,k,m$ be positive integers.
		Then, for every directed graph $\overrightarrow{D}$ with $\chi(D)\ge c+2(m-1)k$, there exists a partition $\bigcup_{i=1}^m C_i$ of $V(D)$ such that $\chi(D[C_m]) \ge c$ and for every $2\le  j \le m$ and $v\in \bigcup_{i=j}^m C_i$, we have that $\left| N^{+}(v) \cap \left( \bigcup_{i=j-1}^{m} C_i \right) \right| \ge k$.
	\end{lemma}
	
	\begin{proof}
		The lemma is trivially true if $m=1$, so we shall proceed inductively assuming the lemma holds for $m-1$.
		By the inductive hypothesis, there exists a partition $\bigcup_{i=1}^{m-1} C_i'$ of $V(D)$ such that $\chi(D[C_{m-1}']) \ge c + 2k$ and for every $2\le  j \le m-1$ and $v\in \bigcup_{i=j}^{m-1} C_i'$, we have that $\left| N^{+}(v) \cap \left( \bigcup_{i=j-1}^{m-1} C_i' \right) \right| \ge k$.
		For each $1\le i \le m-2$, let $C_i = C_i'$.
		Now, let $C_{m-1}$ be the vertices of $C_{m-1}'$ with out-degree at most $k-1$ in $\overrightarrow{D[C_{m-1}']}$, and let $C_m = C_{m-1}' \backslash C_{m-1}$.
		Then by definition, for all $v\in C_m$, we have that $\left| N^{+}(v) \cap \left( \bigcup_{i=m-1}^{m} C_i \right) \right| \ge k$, and clearly for every $2\le  j \le m-1$ and $v\in \bigcup_{i=j}^k C_i$, we also still have that $\left| N^{+}(v) \cap \left( \bigcup_{i=j-1}^{m} C_i \right) \right| \ge k$.
		Since $\overrightarrow{D[C_{m-1}]}$ is $(2(k-1))$-degenerate, we have $\chi(D[C_{m-1}]) \le 2k$.
		Therefore, $\chi(D[C_m]) \ge \chi(D[C_{m-1}']) - \chi(D[C_{m-1}]) \ge c$, as desired.
	\end{proof}
	
	A directed graph $\overrightarrow{D}$ is \emph{acyclic} if it contains no directed cycles.
	We say that $\overrightarrow{D}$ is \emph{modest} if for every directed edge $uv$ and for every directed path $\overrightarrow{P}$ from $u$ to $v$, $V(P)$ is a clique.
	For a directed graph $\overrightarrow{G}$, we say that a directed subgraph $\overrightarrow{H}$ is \emph{path-induced} if for every directed path $\overrightarrow{P}$ of $\overrightarrow{H}$, $\overrightarrow{P}$ is furthermore induced in $\overrightarrow{G}$.
	Similarly for $\overrightarrow{H}$ when $H$ is a subgraph of an (undirected) graph $G$.
	
	We can now obtain the main result of this section.
	This will be used in the proof of \Cref{main} in the next section.
	
	\begin{lemma}\label{directedmain}
		Let $\overrightarrow{T}$ be a rooted directed tree with depth $d$ and $n$ vertices.
		Then every acyclic modest directed graph $\overrightarrow{D}$ with $\chi(D) > 2dn\omega(D)$ contains a path-induced copy of $\overrightarrow{T}$.
	\end{lemma}
	
	\begin{proof}
		By \Cref{highDegree}, there exists a partition $\bigcup_{i=1}^{d\omega(D)} C_i$ of $V(D)$ such that $\chi(D[C_{d\omega(D)}]) \ge 1$ and for every $2\le  j \le {d\omega(D)}$ and $v\in \bigcup_{i=j}^{d\omega(D)} C_i$, we have that $\left| N^{+}(v) \cap \left( \bigcup_{i=j-1}^{d\omega(D)} C_i \right) \right| \ge n$.
		For each $1\le i \le {d\omega(D)}$, let $X_i=\bigcup_{j=d\omega(D)-i+1}^{d\omega(D)} C_i$.
		Then, $(X_1, \ldots , X_{d\omega(D)})$ is a $(n,{d\omega(D)})$-grading of $D$.
		By \Cref{grading}, there is a calm copy $\overrightarrow{T'}$ of $\overrightarrow{T}$ in $(X_1, \ldots , X_{d\omega(D)})$. It remains to show that $\overrightarrow{T'}$ is path-induced in $\overrightarrow{D}$.
		
		Suppose for sake of contradiction that $\overrightarrow{T'}$ is not path-induced.
		Then there is a directed path $x_1x_2\cdots x_q$ of $\overrightarrow{T'}$ such that $x_1$ and $x_q$ are adjacent in $D$. Since $D$ is acyclic, we have that $x_1x_q\in E(\overrightarrow{D})$.
		As $\overrightarrow{D}$ is modest, it follows that $x_1x_3\in E(\overrightarrow{D})$.
		By the definition of a calm copy of $\overrightarrow{T}$ in $(X_1, \ldots , X_{d\omega(D)})$, there is a transitive tournaments $A_1$ starting at $x_1$ and ending at $x_2$ where $|A_1|=t(x_1,x_2) \ge g(x_2) - g(x_1) + 1$, and similarly, there is a transitive tournaments $A_2$ starting at $x_2$ and ending at $x_3$ where $|A_2|=t(x_2,x_3) \ge g(x_3) - g(x_2) + 1$.
		Clearly, $|A_1|,|A_2|\ge 2$.
		Since $\overrightarrow{D}$ is acyclic, we have that $V(A_1)\cap V(A_2) = \{x_2\}$.
		As $x_1x_3\in E(\overrightarrow{D})$ and $\overrightarrow{D}$ is modest, it follows that there is a transitive clique $A_3$ with vertex set $V(A_1)\cup V(A_2)$ between $x_1$ and $x_3$.
		We have that
		\[
		|A_3|=|A_1|+|A_2|-1
		\ge (g(x_2)-g(x_1) + 1)
		+ (g(x_3)-g(x_2) + 1)
		- 1
		=
		g(x_3)-g(x_1) + 1.
		\]
		But then, $t(x_1,x_3) - 1 \ge g(x_3) - g(x_1)$ and $t(x_1,x_3) > t(x_1,x_2)$, contracting the fact that $\overrightarrow{T'}$ is a calm copy of $\overrightarrow{T}$ in $(X_1, \ldots , X_{d\omega(D)})$.
		Hence $\overrightarrow{T'}$ is a path-induced copy of $\overrightarrow{T}$, as desired.
	\end{proof}

	\section{Box intersection graphs}\label{sec:boxes}
	
	In this section we shall prove \Cref{main}.
	In addition to \Cref{directedmain}, we shall require two more preliminary lemmas first.
	
	For a collection $\mathcal{B}$ of axis-aligned boxes in $\mathbb{R}^d$, its intersection graph is the graph with vertex set $\mathcal{B}$ where $B_1,B_2\in \mathcal{B}$ are adjacent if and only if they intersect.
	For $1\le i \le d$, we let $p_i:\mathbb{R}^d \to \mathbb{R}$ be the projection onto the $i$-th coordinated axis.
	
	For our proof we require the Erd\H{o}s-Hajnal property for box intersection graphs. We remark that some much more general classes of graphs, such as graphs of bounded VC-dimension are known to have the Erd\H{o}s-Hajnal property~\cite{nguyen2023induced}.
	For completeness, we give a simple inductive proof for box intersection graphs with explicit bounds.
	
	\begin{lemma}\label{boxEH}
		Let $\mathcal{B}$ be a collection of axis-aligned boxes in $\mathbb{R}^d$ and let $G$ be their intersection graph. Then $|V(G)|\le  \alpha(G)^d 
		\omega(G)$.
	\end{lemma}
	
	\begin{proof}
		The lemma holds for $d=1$, since then $G$ is an interval graph and therefore $\omega(G)$-colourable.
		We proceed inductively with $d\ge 2$.
		Consider the intersection graph $H$ of $p_d(\mathcal{B})$.
		Then $|V(G)|=|V(H)| \le \alpha(H)\omega(H)$.
		Let $X$ be the vertex set of a clique in $H$ with size $\omega(H)$.
		Let $p^*: \mathbb{R}^d \to \mathbb{R}^{d-1}$ be the projection onto the subspace spanned by the first $d-1$ axes.
		Let $H'$ be the intersection graph of $p^*(\mathcal{B})$ restricted to the vertex set $X$.
		Since $X$ is a clique in $H$, it follows that a clique in $H'$ is a clique in $G$. So, $\omega(G)\ge \omega(H')$. We also clearly have that $\alpha(G) \ge \alpha(H)$ and $\alpha(G) \ge \alpha(H')$.
		Therefore,
		\[
		|V(G)|
		=
		|V(H)|
		\le
		\alpha(H)\omega(H)
		\le
		\alpha(G)|V(H')|
		\le
		\alpha(G)\alpha(H')^{d-1}\omega(H')
		\le
		\alpha(G)^d \omega(G),
		\]
		as desired.
	\end{proof}

	For an interval $I\subset \mathbb{R}$, we let $\ell(I)$ denote its leftmost point and $r(I)$ denote its rightmost point.
	Notice that for a box intersection graph $G(\mathcal{B})$, by possibly slightly perturbing the boxes of $\mathcal{B}$, we can assume that there are no two distinct boxes $B_1, B_2 \in \mathcal{B}$ such that $p_i(B_1), p_i(B_2)$ share an endpoint. For convenience, we shall assume this for the rest of the paper.

	For two intersecting intervals $I_1,I_2$, we say that
	\begin{itemize}
		\item $I_1$ \emph{contains} $I_2$ if $\ell(I_1) < \ell(I_2) < r(I_2) < r(I_1)$ (or equivalently $I_2\subset I_1$),
		\item $I_1$ is \emph{contained} in $I_2$ if $\ell(I_2) < \ell(I_1) < r(I_1) < r(I_2)$ (or equivalently $I_1\subset I_2$),
		\item $I_1$ \emph{left-intersects} $I_2$ if $\ell(I_1)<\ell(I_2) < r(I_1) < r(I_2)$, and 
		\item $I_1$ \emph{right-intersects} $I_2$ if $\ell(I_2)<\ell(I_1) < r(I_2) < r(I_1)$.
	\end{itemize}
	Notice that two intersecting intervals with disjoint endpoints intersect with one of these four types.
	If two axis-aligned boxes $B_1,B_2\subset \mathbb{R}^d$ intersect, then for each $1\le i \le d$, the intervals $p_i(B_1),p_i(B_2)$ intersect with one of these four types.
	
	For any pair of axis-aligned intersecting boxes $B_1,B_2\subset \mathbb{R}^d$ we associate a $d$-tuple which encodes the different types of intersections of the projections of the boxes onto each of the $d$ axes.
	We call this the \emph{intersection pattern} of $B_1$ with respect to $B_2$.
	We let $\mathcal{P}_d$ be the possible intersection patterns of axis-aligned boxes in $\mathbb{R}^d$.
	Note that $|\mathcal{P}_d|=4^d$.
	If $G$ is the intersection graph of a collection of axis-aligned boxes in $\mathbb{R}^d$, and $R\in \mathcal{P}_d$ is an intersection pattern, then we let $\overrightarrow{G_R}$ be the directed graph on the same vertex set where $uv\in E(\overrightarrow{G_R})$ if the corresponding box $B_u$ has the intersection pattern $R$ with respect to the corresponding box $B_v$.
	Note that $G_R$ is a subgraph of $G$.
	
	We require some graph theoretic properties of $\overrightarrow{G_R}$ in relation to $G$.
	
	\begin{lemma}\label{basic}
		Let $\mathcal{B}$ be a collection of axis-aligned boxes in $\mathbb{R}^d$ and let $G$ be their intersection graph.
		Let $R\in \mathcal{P}_d$ be an intersection pattern.
		Then $\overrightarrow{G_R}$ is acyclic, modest, and furthermore for every pair of directed paths $ux_1\cdots x_a$, $uy_1\cdots y_b$ in $\overrightarrow{G_R}$, if $x_1$ is non-adjacent to $y_1$ in $G$, then $x_a$ is non-adjacent to $y_a$ in $G$.
	\end{lemma}
	
	\begin{proof}
		First we shall show that $\overrightarrow{G_R}$ is acyclic.
		The intersection pattern $R$ ensures that either for every edge $uv$ we have that $\ell(p_1(B_u)) < \ell(p_1(B_v))$, or for every edge $uv$ we have that $\ell(p_1(B_v)) < \ell(p_1(B_u))$.
		Without loss of generality, we assume the former.
		Consider an arbitrary directed path $u_1u_2\cdots u_k$ in $\overrightarrow{G_R}$.
		Then, $\ell(p_1(B_{u_1})) < \ell(p_1(B_{u_2})) < \cdots < \ell(p_1(B_{u_k}))$.
		As $\ell(p_1(B_{u_1})) <  \ell(p_1(B_{u_k}))$, it follows that there is no edge from $u_k$ to $u_1$ in $\overrightarrow{G_R}$.
		Therefore, $\overrightarrow{G_R}$ is acyclic.
		
		Next, we show that $\overrightarrow{G_R}$ is modest.
		Consider an arbitrary directed path $u_1u_2\cdots u_k$ in $\overrightarrow{G_R}$ with $u_1u_k\in E(\overrightarrow{G_R})$.
		To show that $C=\{u_1, \ldots , u_k\}$ is a clique, it is enough to show that for each $1\le j \le d$, there is a point $s_j\in \mathbb{R}$ contained in each of $p_j(B_{u_1}) , \ldots , p_j(B_{u_k})$, since then $B_{u_1}, \ldots , B_{u_k}$ have a common intersection point. So, consider some $1\le j \le d$ and the intersection type $R_j$ of $R$ in the $j$-coordinate. If $R_j$ is either the contains or right-intersects intersection types, then we can take $s_j = \ell(p_j(B_{u_k}))$. Otherwise, if $R_j$ is either the containment or left-intersects intersection types, then we can take $s_j = r(p_j(B_{u_1}))$.
		Therefore, $C$ is a clique.
		Hence, $\overrightarrow{G_R}$ is modest.
		
		Now, for the last property.
		Consider two directed paths $ux_1\cdots x_a$, $uy_1\cdots y_b$ in $\overrightarrow{G_R}$ with $x_1$ non-adjacent to $y_1$ in $G$.
		Since $x_1$ and $y_1$ are non-adjacent in $G$ and both have the same intersection pattern with $u$, it follows that for some $1\le j \le d$, $R_j$ is the contains relation and that $p_j(B_{x_1}), p_j(B_{y_1})$ are disjoint.
		Then, we further have that $p_j(B_{x_a}) \subset p_j(B_{x_{a-1}}) \subset \cdots \subset  p_j(B_{x_1})$ and $p_j(B_{y_b}) \subset p_j(B_{y_{b-1}}) \subset \cdots \subset p_j(B_{y_1})$.
		In particular, $p_j(B_{x_a})$ and $p_j(B_{y_b})$ are disjoint, and therefore $x_a$ and $y_b$ are non-adjacent in $G$.
	\end{proof}

	For positive integers $r,k$, we let $\overrightarrow{T_{r,k}}$ be the rooted directed tree such that the paths between the root and any leaf are of length $r$, and each non-leaf vertex $v$ has exactly $k$ out-neighbours.
	
	We are now ready to prove \Cref{main}, which we restate with explicit bounds.
	
	\begin{theorem}
		Let $G$ be an intersection graph of axis-aligned boxes in $\mathbb{R}^d$ with no induced $T_{r,k}$ subgraph.
		Then $\chi(G) \le (2rk^d\omega(G)^2)^{4^{d}}$.
	\end{theorem}
	
	\begin{proof}
		Suppose for sake of contradiction that $\chi(G) > (2rk^d\omega(G)^2)^{4^{d}}$.
		Observe that $G$ is the union of the subgraphs $G_R$, where $R$ is taken across all $4^{d}$ intersection patterns in $\mathcal{P}_d$.
		By considering a product colouring, we see that there is an intersection pattern $R$ such that $\chi(G) > 2rk^d\omega(G)^2$.
		Note that $\overrightarrow{G_R}$ is acyclic and modest by \Cref{basic}.
		
		By \Cref{directedmain}, $\overrightarrow{G_R}$ contains a path-induced $\overrightarrow{T_{r,k^d \omega(G)}}$.
		Observe that $T_{r,k^d \omega(G)}$ must furthermore be path-induced in $G$.
		For every non-leaf vertex $v$ of $T_{r,k^d \omega(G)}$, by \Cref{boxEH} there exists a set $S_v\subseteq N^+_{\overrightarrow{T_{r,k^d \omega(G)}}}(v)$ that is independent in $G$ and with $|S_v|=k$.
		Let $L_0$ be the vertex set consisting simply of the root of $\overrightarrow{T_{r,k^d \omega(G)}}$.
		For each $1\le i \le r$ in order, let $L_i=\bigcup_{v\in L_{i-1}} S_v$.
		Then restricting $\overrightarrow{T_{r,k^d \omega(G)}}$ to the vertex set $\bigcup_{i=0}^r L_i$ provides an induced $T_{r,k}$ in $G$ by \Cref{basic}.
		
		Hence, $\chi(G) \le (2rk^d\omega(G)^2)^{4^{d}}$, as desired.
	\end{proof}

	\section{Open problems}\label{sec:open}
	
	In this section we discuss some open problems on proving the Gyárfás-Sumner conjecture for other classes of graphs that also include graphs with large girth and chromatic number.
	In this paper we have proved the Gyárfás-Sumner conjecture for box intersection graphs in $\mathbb{R}^d$, and
	as previously mentioned, R\"{o}dl (see~\cite{gyarfas1980induced,kierstead1996applications}) proved the Gyárfás-Sumner conjecture for graphs without an induced $K_{s,s}$.
	So, we shall only consider classes of graphs that contain large bicliques.
	
	As shown by the first author~\cite{davies2021box}, another natural class of geometric intersection graphs that contain graphs with large girth and chromatic number are the intersection graphs of lines in $\mathbb{R}^3$.
	We remark that in $\mathbb{R}$, line intersection graphs are simply cliques, and in $\mathbb{R}^2$, line intersection graphs are simply complete multipartite graphs.
	
	\begin{problem}
		Does the Gyárfás-Sumner conjecture hold for intersection graphs of lines in $\mathbb{R}^d$?
		What about intersection graphs of segments in $\mathbb{R}^d$?
	\end{problem}
	
	For a collection of $d$-spheres $\mathcal{S}$ in $\mathbb{R}^{d+1}$, we say that the graph with vertices $\mathcal{S}$ and edges $S_1S_2$ when $S_1$ and $S_2$ are orthogonal $d$-spheres is the \emph{orthogonality graph} of $\mathcal{S}$.
	Recently, the first author, Keller, Kleist, Smorodinsky, and Walczak~\cite{davies2021solution} showed that there are orthogonality graphs of circles in $\mathbb{R}^2$ with large girth and chromatic number.
	
	\begin{problem}
		Does the Gyárfás-Sumner conjecture hold for orthogonality graphs of $d$-spheres?
	\end{problem}
	
	Bollob{\'a}s~\cite{bollobas1977colouring} showed that there are Hasse diagrams with large girth and chromatic number (see also~\cite{suk2021hasse} for another construction).
	A natural geometric class of graphs that contains Hasse diagrams are the disjointness graphs of strings in the plane~\cite{middendorf1993weakly} (or equivalently, the complements of string graphs).
	
	\begin{problem}
		Does the Gyárfás-Sumner conjecture hold for Hasse diagrams?
		What about disjointness graphs of strings in the plane?
	\end{problem}

	While not every hereditary $\chi$-bounded class of graphs is polynomially $\chi$-bounded~\cite{brianski2024separating}, it is true that some (not necessarily $\chi$-bounded) hereditary classes of graphs have the property that every hereditary $\chi$-bounded subclass of graphs is polynomially $\chi$-bounded.
	Such classes are called \emph{Pollyanna} and have recently been studied by Chudnovsky, Cook, Davies, and Oum~\cite{chudnovsky2023reuniting}.
	In \Cref{main}, we obtained polynomial bounds for the $\chi$-bounding functions.
	Perhaps much more generally, box intersection graphs could be Pollyanna.
	Note that interval graphs are perfect and rectangle intersection graphs are polynomially $\chi$-bounded~\cite{asplund1960coloring,chalermsook2021coloring}, so this is certainly true in at most $2$ dimensions.
	
	\begin{conjecture}
		Intersection graphs of axis-aligned boxes in $\mathbb{R}^d$ are Pollyanna.
	\end{conjecture}
	
	\section*{Acknowledgements}
	
	This work was completed at the 11th Annual Workshop on Geometry and Graphs held at Bellairs Research Institute in March 2024.
	We are grateful to the organisers and participants for providing a stimulating research environment. Yelena Yuditsky was supported by FNRS as a Postdoctoral Researcher.

	\bibliographystyle{amsplain}

\end{document}